\documentclass[12pt]{article}

\usepackage{amsmath,amsfonts,amssymb,latexsym,amsthm,verbatim,a4wide}
\usepackage{amsfonts}
\usepackage{amsmath}
\usepackage{amssymb}
\usepackage{color}
\usepackage{amsthm}

 \theoremstyle{plain}
  \newtheorem{theorem}{Theorem}[section]
  \newtheorem{corollary}[theorem]{Corollary}
  \newtheorem{proposition}[theorem]{Proposition}
  \newtheorem{lemma}[theorem]{Lemma}  
  \newtheorem{definition}[theorem]{Definition}
  
  \newtheorem{remark}[theorem]{Remark}
 \newtheorem{example}[theorem]{Example}

\newcommand{\II}{{\mathbb{I}}}
\newcommand{\ol}[1]{\overline{#1}}

\newcommand{\ep}{{\mathbb {E}}}
\newcommand{\pr}{{\mathbb {P}}}

\newcommand{\blah}[1]{}

\newcommand{\diy}{\begin{displaystyle}}
\newcommand{\eiy}{\end{displaystyle}}

\title{Relaxation of monotone coupling conditions: Poisson approximation and beyond}
\author{Fraser Daly\footnote{Department of Actuarial Mathematics and Statistics and the Maxwell Institute for Mathematical Sciences, Heriot-Watt University, Edinburgh EH14 4AS, UK.  E-mail: f.daly@hw.ac.uk;  Tel: +44 (0)131 451 3212; Fax: +44 (0)131 451 3249} \and Oliver Johnson\footnote{School of Mathematics, University of Bristol, University Walk, Bristol BS8 1TW, UK. Email: o.johnson@bristol.ac.uk; Tel: +44 (0)117 928 8632}}
\date{\today}

\begin{document}

\maketitle

\noindent{\bf Abstract} It is well-known that assumptions of monotonicity in size-bias couplings may be used to prove simple, yet powerful, Poisson approximation results.  Here we show how these assumptions may be relaxed, establishing explicit Poisson approximation bounds (depending on the first two moments only) for random variables which satisfy an approximate version of these monotonicity conditions.  These are shown to be effective for models where an underlying random variable of interest is contaminated with noise.  We also give explicit Poisson approximation bounds for sums of associated or negatively associated random variables.  Applications are given to epidemic models, extremes, and random sampling.  Finally, we also show how similar techniques may be used to relax the assumptions needed in a  Poincar\'e inequality and in a normal approximation result.

\vspace{12pt}
\vspace{12pt}

\section{Introduction}

It is well-known   that in many situations exploiting negative or positive dependence structure  is an effective way to establish Poisson approximation results.  For example, Barbour, Holst and Janson \cite{barbour} treat many applications of Poisson approximation for sums of (dependent) Bernoulli random variables $X_1,X_2,\ldots,X_n$ which are negatively related, that is, satisfy
\begin{equation}\label{eq:nr}
\mathbb{E}[\phi(X_1,\ldots,X_{i-1},X_{i+1},\ldots,X_n)|X_i=1]\leq \mathbb{E}[\phi(X_1,\ldots,X_{i-1},X_{i+1},\ldots,X_n)]\,,
\end{equation}
for all $i=1,\ldots,n$ and increasing functions $\phi:\{0,1\}^{n-1}\mapsto\{0,1\}$.

The Poisson approximation bounds given in \cite{barbour} under this negative relation assumption have the advantage of only depending on the first two moments of $W$.  In  general, such bounds require much more detailed information about the $X_i$ in order to be evaluated.  More precisely, let $d_{TV}$ be the total variation distance, defined for non-negative integer-valued random variables $W$ and $Z$ by
$$
d_{TV}(\mathcal{L}(W),\mathcal{L}(Z))=\sup_{A\subseteq\mathbb{Z}^+}|\mathbb{P}(W\in A)-\mathbb{P}(Z\in A)|\,,
$$   
where $\mathbb{Z}^+=\{0,1,\ldots\}$.  Then if $W$ is a sum of negatively related Bernoulli random variables $X_1,\ldots,X_n$ with $\lambda=\mathbb{E}W$, and $Z\sim\mbox{Po}(\lambda)$ has a Poisson distribution with mean $\lambda$, then Corollary 2.C.2 of \cite{barbour} gives the bound
\begin{equation}\label{eq:tvub}
d_{TV}(\mathcal{L}(W),\mathcal{L}(Z))\leq(1-e^{-\lambda})\left(1-\frac{\mbox{Var}(W)}{\lambda}\right)\,.
\end{equation}
This upper bound is considerably simpler to evaluate in practice than more general Poisson approximation bounds, many of which involve, for example, $\mbox{Cov}(X_i,X_j)$ for each $i$ and $j$.  Thus, if negative relation can be shown to hold, it is worthwhile taking advantage of it.  Straightforward Poisson approximation bounds are also available 
when $X_1,\ldots,X_n$ are positively related, a property which is defined analogously to (\ref{eq:nr}), but with the inequality reversed.  

Within a more general approximation framework, Daly, Lef\`evre and Utev \cite{dlu12} recently showed that the upper bound of (\ref{eq:tvub}) continues to hold when $W$ is a non-negative, integer-valued random variable with mean $\lambda$, under the assumption that
\begin{equation}\label{eq:ord}
W+1\geq_{st}W^*\,,
\end{equation}
where $W^*$ has the $W$-size-biased distribution, defined by
\begin{equation} \label{eq:sizebias}
\mathbb{P}(W^*=j)=\frac{j\mathbb{P}(W=j)}{\mathbb{E}W}\,,
\end{equation}
and `$\geq_{st}$' denotes the usual stochastic ordering.  If $W$ is a sum of negatively related Bernoulli random variables, then $W$ is also shown to satisfy (\ref{eq:ord}), so (\ref{eq:ord}) is referred to as a negative dependence condition for $W$. Daly, Lef\`evre and Utev \cite{dlu12} also show simple Poisson approximation results under an analogous positive dependence condition expressed in terms of monotonicity of the size-biased coupling.  

Daly and Johnson \cite{dj13}   proved a simple Poincar\'e inequality for $W$ under the same condition (\ref{eq:ord}), (again, with an upper bound depending only on the first two moments of $W$).  Boundedness and monotonicity conditions of a similar type are also exploited in several of the normal approximation theorems discussed by Chen, Goldstein and Shao \cite{cgs11}. 

Our aim in this work is to demonstrate how the strict condition of assumptions such as (\ref{eq:ord}) may be relaxed.
The structure of the remainder of the paper is as follows. In Section \ref{sec:pois} we derive  Poisson approximation bounds (Theorem \ref{thm:po}) for random variables $W$ which come close (in a certain sense) to satisfying either (\ref{eq:ord}) or its positive dependence analogue.  This relaxation of these conditions is motivated by recent work of Cook, Goldstein and Johnson \cite{cgj15}, who established a concentration inequality which can control the spectral gap of random graphs. 
  In Section \ref{sec:noise} we use Theorem \ref{thm:po}  to derive explicit Poisson approximation results for models where an underlying random variable of interest $W$, satisfying (\ref{eq:ord}), or its positive dependence analogue, is contaminated by noise. In Section \ref{sec:applications}, we show how these results can be applied to the Martin-L\"of epidemic model and extremes of associated random variables.

In Section \ref{sec:assoc} we consider Poisson approximation for sums of associated or negatively associated random variables, illustrated by an application to simple random sampling. 
In Section \ref{sec:poin}, we show how such relaxed conditions imply  a  Poincar\'e inequality.   Finally, in Section \ref{sec:norm} we 
give an analogous relaxation of strict boundedness or monotonicity conditions for continuous random variables, to obtain normal approximation results.
\section{Poisson approximation results}\label{sec:pois}
In this section, we begin by proving a Poisson approximation bound (Theorem \ref{thm:po} below) under conditions in the spirit of \cite{cgj15}, which relax strict monotonicity assumptions such as (\ref{eq:ord}).  As we will see, the upper bounds we obtain have a form similar to that of (\ref{eq:tvub}).  The condition (\ref{eq:ndcond}) below is employed directly by \cite{cgj15}, where a random variable $Y$ satisfying this condition is said to be `$1$-bounded with probability $p$ for the upper tail'.  The analogous condition (\ref{eq:pdcond}) was not used by \cite{cgj15}, but is in the same spirit.

We then examine situations in which these Poisson approximation results may be applied.  Section \ref{sec:noise} considers models where we have an underlying random variable $W$, satisfying a property such as (\ref{eq:ord}), which is then contaminated with  independent noise.  

\subsection{A Poisson approximation theorem}

Our main Poisson approximation result is given in Theorem \ref{thm:po} below, and is proved using the Stein--Chen method.  For an introduction to the Stein--Chen method for Poisson approximation, the interested reader is referred to \cite{barbour}, \cite{e05} and references therein.  

Throughout this section, given a parameter $\lambda\geq0$ and a set $A\subseteq\mathbb{Z}^+$, we write $g_A$ for the solution to the Stein--Chen equation
\begin{equation} \label{eq:stein}
\lambda g_A(j+1) - j g_A(j) = \II( j \in A) - \Pi_{\lambda}(A),
\end{equation}
where $\Pi_\lambda(A)=\mathbb{P}(Z\in A)$ and $Z\sim\mbox{Po}(\lambda)$.  By convention, we take $g_A(0)=0$ for each $A$.  
Writing $\Delta f(x) = f(x+1) - f(x)$ for any function $f$, we recall the standard bound \cite[Eq. (1.17)]{barbour}
\begin{equation}\label{eq:magic}
\sup_x \left| \Delta g_A(x) \right| \leq \lambda^{-1}(1- e^{-\lambda})\,.
\end{equation}

For any non-negative, integer-valued random variable $Y$, we write the upper tail
$$
\ol{P}_Y(y) = \sum_{z=y}^\infty \mathbb{P}(Y=z)\,.
$$
\begin{lemma} \label{lem:key}
For any non-negative, integer-valued random variable $Y$ with $\mathbb{E}Y = \mu >0$
\begin{equation} \label{eq:setprob}
\mathbb{P}( Y \in A) - \Pi_\lambda(A) =  \sum_{k=0}^\infty \left( \lambda \ol{P}_Y(k)    - \mu  \ol{P}_{Y^*}(k+1)  \right) \Delta g_A(k)\,,
\end{equation}
for any $\lambda>0$. Here, and for the rest of the paper, $Y^*$ represents the size-biased version of $Y$ defined in \eqref{eq:sizebias}.
\end{lemma}
The proof of Lemma \ref{lem:key} is deferred until Appendix \ref{sec:lem}.  We now apply this representation to prove the following.
\begin{theorem}\label{thm:po}
Let $Y$ be a non-negative, integer-valued random variable with $\mathbb{E}Y=\mu>0$ and $\mbox{Var}(Y)=\sigma^2$.
\begin{enumerate}
\item[(i)] Suppose there is a coupling of $(Y,Y^*)$ and $p\in(0,1]$ such that 
\begin{equation}\label{eq:ndcond}
\mathbb{P}(Y^*\leq Y+1|Y^*\geq x)\geq p\,,
\end{equation}
for all $x$.  Then, for any $\lambda>0$,
\begin{equation} \label{eq:negdepbd}
d_{TV}(\mathcal{L}(Y),\mbox{Po}(\lambda))\leq (1-e^{-\lambda})\left\{1+\mu+\left(\frac{|\mu-p\lambda|}{\lambda}-p\right)\left(\frac{\sigma^2}{\mu}+\mu\right)\right\}\,.
\end{equation}
\item[(ii)] Suppose there is a coupling $(Y,Y^*)$, $p\in(0,1]$, and a non-negative integer-valued random variable $Z$ (which may be dependent on $(Y,Y^*)$) such that
\begin{equation}\label{eq:pdcond}
\mathbb{P}(Y^*\geq Y+1-Z|Y+1-Z\geq x)\geq p\,,
\end{equation}
for all $x$.  Then for any $\lambda>0$
\begin{multline*}
d_{TV}(\mathcal{L}(Y),\mbox{Po}(\lambda))\\ \leq
(1-e^{-\lambda})\left\{2p\mathbb{E}Z+\left(\frac{|\mu-\lambda|}{\lambda}+1\right)\left(\frac{\sigma^2}{\mu}+\mu\right)+(1-2p)(\mu+1)\right\}\,.
\end{multline*}
\end{enumerate}
\end{theorem} 
\begin{proof}
\begin{enumerate}
\item[(i)] Under the assumptions of the first part of the theorem,
\begin{multline}\label{eq:tails}
\mathbb{P}(Y^*\geq x)=\frac{\mathbb{P}(Y^*\leq Y+1\mbox{ and }Y^*\geq x)}{\mathbb{P}(Y^*\leq Y+1 | Y^*\geq x)}\\
\leq\frac{\mathbb{P}(Y+1\geq x)}{\mathbb{P}(Y^*\leq Y+1 | Y^*\geq x)}\leq\frac{1}{p}\mathbb{P}(Y+1\geq x)\,.
\end{multline}
Note that \eqref{eq:tails}  is equivalent to the stochastic ordering $I_pY^*\leq_{st}Y+1$, where $I_p\sim\mbox{Be}(p)$ is  Bernoulli with mean $p$ independent of all else,
and hence generalizes \eqref{eq:ord}.  This stochastic ordering assumption was considered by \cite{dlu12}, and the upper bound of part (i) now follows from their Proposition 3.  For completeness we give a self-contained proof here.  An analogous proof will  be used for part (ii), in the case of positive dependence, for which no corresponding bound is available elsewhere. 

We now apply the representation of Lemma \ref{lem:key}.  Taking modulus signs and using the triangle inequality and (\ref{eq:magic}), we have that
\begin{eqnarray*}
\left| \pr( Y \in A) - \Pi_\lambda(A) \right|
& \leq &  \sum_{k=0}^\infty \left| \lambda \ol{P}_Y(k)    - \mu  \ol{P}_{Y^*}(k+1)  \right| \left| (\Delta g_A)(k) \right| \\
& \leq &  c_\lambda \left\{
\sum_{k=0}^\infty \left|  \ol{P}_Y(k)    -  p  \ol{P}_{Y^*}(k+1)  \right| + \frac{|\mu -p\lambda|}{\lambda}\sum_{k=0}^\infty \ol{P}_{Y^*}(k+1)\right\}\\
& = &  c_\lambda \left\{
\sum_{k=0}^\infty \left( \ol{P}_Y(k)    -  p  \ol{P}_{Y^*}(k+1)  \right)+\frac{|\mu -p\lambda|}{\lambda}\sum_{k=0}^\infty \ol{P}_{Y^*}(k+1)\right\} \\
& = &  c_\lambda \left\{ 1 + \mu +\left(\frac{|\mu-p\lambda|}{\lambda} - p\right) \ep Y^* \right\}\,,
\end{eqnarray*}
where $c_\lambda=1-e^{-\lambda}$.  Since by \eqref{eq:sizebias} $\mathbb{E}Y^*= \ep Y^2/\ep Y = \frac{\sigma^2}{\mu}+\mu$, the proof is complete.

\item[(ii)] We use a similar argument to the above.  We have
\begin{multline*}
\mathbb{P}(Y+1-Z\geq x)=\frac{\mathbb{P}(Y^*\geq Y+1-Z\mbox{ and }Y+1-Z\geq x)}{\mathbb{P}(Y^*\geq Y+1-Z|Y+1-Z\geq x)}\leq\frac{1}{p}\mathbb{P}(Y^*\geq x)\,,
\end{multline*}
from which an analogous argument to part (i) gives
$$
\left| \pr( Y \in A) - \Pi_\lambda(A) \right|
\leq c_\lambda\left\{2p\mathbb{E}Z+\left(\frac{|\mu-\lambda|}{\lambda}+1\right)\mathbb{E}Y^*+(1-2p)(\mu+1)\right\}\,.
$$
\end{enumerate}
\end{proof}
\begin{remark}
\emph{Taking $p=1$ in Theorem \ref{thm:po}, we recover the results we expect under the stochastic ordering assumptions of \cite{dlu12}.  For example, with $p=1$ and $\lambda=\mu$,  the upper bound of Theorem \ref{thm:po}(i) reduces to  (\ref{eq:tvub}).}
\end{remark}
\begin{example}\label{eg:po}
\emph{Let $Z\sim\mbox{Po}(\lambda)$ and $I_p\sim\mbox{Be}(p)$ be independent.  In the zero-inflated Poisson case where $Y = I_pZ$ and $Y^* = Z+1$, the argument of Example 3.6 of \cite{cgj15} shows that we may apply our Theorem \ref{thm:po}(i), with the $p$ and $\lambda$ we have defined here. Indeed, direct calculation gives
$$ \pr( Y^* \leq Y+1 | Y^* \geq x) = \pr( Z + 1 \leq I_p Z + 1 | Z+1 \geq x) = \pr(1 \leq I_p) = p,$$
so that \eqref{eq:ndcond} holds with equality.
The bound \eqref{eq:negdepbd} is $(1- e^{-\lambda}) ( 1 + p \lambda - p(1+\lambda)) = (1-p)(1- e^{-\lambda})$.
Since we may choose $A = \{ 0 \}$ and obtain
$$ \mathbb{P}(Y=0) - \Pi_{\lambda}(0) = (1-p) + p e^{-\lambda} - e^{-\lambda} = (1-p)(1- e^{-\lambda})\,,$$
the bound \eqref{eq:negdepbd} is  exact  in this case.}
\end{example}

\subsection{Models contaminated with noise}\label{sec:noise}

In this section, we show how the assumptions of Theorem \ref{thm:po} are satisfied by a random variable $Y$, made up of a random variable $W$ which satisfies a monotone coupling assumption such as (\ref{eq:ord}), together with some independent noise.  In this case, we expect that $Y$ will be close to Poisson as long as $W$ is close to Poisson and the noise is small.  This is confirmed in the explicit bounds we derive below.

We consider separately the cases corresponding to parts (i) and (ii) of Theorem \ref{thm:po}, beginning with part (i), the negatively dependent case. 

Consider first the random variable $Y=\xi W+X$, where
\begin{itemize}
\item $W$ and $W^*$ can be coupled such that $W+1\geq W^*$ almost surely.  Note that this is possible if $W$ satisfies (\ref{eq:ord}), which holds, for example, if $W$ may be written as a sum of negatively related Bernoulli random variables (see  \cite{barbour} for applications where this situation arises naturally).
We let $\mathbb{E}W=\nu$.
\item $X$ is a non-negative, integer-valued random variable independent of $(W,W^*)$ with mean $\phi$.
\item $\xi\sim\mbox{Be}(q)$ is a Bernoulli random variable, independent of all else.
\end{itemize}
\begin{theorem}\label{thm:negdep}
With this choice of $Y$, we may apply Theorem \ref{thm:po}(i) with $p=\frac{q^2\nu}{q\nu+\phi}$.
\end{theorem}
\begin{proof}
Following, for example, Corollary 2.1 of \cite{cgs11}, we may construct $Y^*$ by replacing either $\xi W$ or $X$ by its size-biased version, with the term to replace chosen with probability proportional to its mean.  We obtain
\begin{equation}\label{eq:sb}
Y^*=I(W^*+X)+(1-I)(\xi W+X^*)\,,
\end{equation}
since $(\xi W)^*$ and $W^*$ are equal in law, where $X^*$ may be constructed independently of $(W,W^*)$, and $I$ is a Bernoulli random variable, independent of all else, with $\mathbb{P}(I=1)=\frac{q\nu}{q\nu+\phi}$.

For any event $B$ and indicator variable $J$, we know that  
\begin{equation} \label{eq:condres}
\pr(B) = \pr(J=0) \pr(B | J = 0)   +   \pr(J=1) \pr(B  | J =1) \geq \pr(J=1) \pr(B | J=1) . \end{equation}
Using this result to condition firstly on $I$ and then on $\xi$, we have
\begin{eqnarray*}
\mathbb{P}(Y^*\leq Y+1|Y^*\geq x)&\geq&\frac{q\nu}{q\nu+\phi}\mathbb{P}(W^*+X\leq Y+1|W^*+X\geq x)\\
&\geq&\frac{q^2\nu}{q\nu+\phi}\mathbb{P}(W^*+X\leq W+1+X|W^*+X\geq x)\\
&=&\frac{q^2\nu}{q\nu+\phi}\,,
\end{eqnarray*}
by our assumptions on $W$, hence \eqref{eq:ndcond} is satisfied with $p$ taking this value.
\end{proof}

Now, to demonstrate the application of Theorem \ref{thm:po}(ii) in this context, we will consider the random variable $Y=\xi W+X$, where
\begin{itemize}
\item there exists a random variable $Z$ (which may depend on $W$ and $W^*$) such that $W+1-Z\leq_{st}W^*$.  In the case where $W$ is a sum of positively related Bernoulli random variables, such a random variable $Z$ exists (see \cite{dlu12}).  Again, the reader is referred to \cite{barbour} for a wealth of applications involving such sums.  We assume we have constructed $(W,Z,W^*)$ in such a way that $W+1-Z\leq W^*$ almost surely.  This is possible under our stochastic ordering assumption.
\item X, $\xi$ and $\nu$ are as above.
\end{itemize}
\begin{theorem}\label{thm:podep}
With this choice of $Y$, we may apply Theorem \ref{thm:po}(ii) with $p=\frac{q^2\nu}{q\nu+\phi}$.
\end{theorem}
\begin{proof}
We still have the representation (\ref{eq:sb}) for $Y^*$.  Proceeding as before, using \eqref{eq:condres}, by conditioning firstly on $I$ and then on $\xi$, we have
\begin{eqnarray*}
&&\mathbb{P}(Y^*\geq Y+1-Z|Y+1-Z\geq x)\\
&&\qquad\geq\frac{q\nu}{q\nu+\phi}\mathbb{P}(W^*+X\geq Y+1-Z|Y+1-Z\geq x)\\
&&\qquad\geq\frac{q^2\nu}{q\nu+\phi}\mathbb{P}(W^*+X\geq W+1-Z+X|W+X+1-Z\geq x)\\
&&\qquad=\frac{q^2\nu}{q\nu+\phi}\,.
\end{eqnarray*}
Hence, (\ref{eq:pdcond}) is satisfied with $p$ taking this value.
\end{proof}

\section{Applications} \label{sec:applications}

\subsection{The Martin-L\"of epidemic model}

We show in this section how our framework implies  a Poisson approximation result for the number of survivors in an epidemic model which is based on the Martin-L\"of \cite{m86} model, but with the addition of an independent `catastrophe' which causes the entire population to become infected.  A Poisson approximation result was derived by Ball and Barbour \cite{bb90} for  the usual Martin-L\"of model, and we base our argument on theirs.

We begin by describing the random graph model used in the construction of the epidemic.  The random directed graph $G$ consists of $n$ vertices. Independently  for each vertex $1\leq i\leq n$, we choose a subset $L_i$ of vertices (distinct from $i$) to connect by a directed edge emanating from vertex $i$.  The value of $N_i=|L_i|$ is chosen from some given distribution, and then, conditional on $N_i=k$, the set $L_i$ is chosen uniformly at random from the $k$-subsets of vertices $j$ with $j\not=i$.

The Martin-L\"of epidemic is then constructed by choosing some initial set $I_0$ of infected vertices; the set of remaining vertices is $S_0$, the set of initial susceptibles.  The epidemic then proceeds in discrete time by recursively defining the set of infected vertices $I_j$ and susceptible vertices $S_j$ at time $j$ using the equations  
$$
I_j=\left(\bigcup_{i\in I_{j-1}}L_i\right)\cap S_{j-1}\,,\mbox{ and }S_j=S_{j-1}\setminus I_j\,,
$$     
for $j\geq1$.  The epidemic ends when $|I_j|=0$ for some $j$.

Ball and Barbour prove a Poisson approximation theorem, with an explicit rate, for $|S_\infty|$, the ultimate number of susceptible vertices remaining, in this model.  We will use their result to prove an analogous result in a modified version of this model.  

To avoid the notational burden associated with the most general version of this model, for most of this section we will concentrate on the Reed--Frost model, in which each $N_i$ has a binomial $\mbox{Bin}(n-1,r)$ distribution.  Following \cite{bb90}, we let $(|I_0|,|S_0|)=(1,n-1)$ and consider the choice $r=\frac{\psi\log(n)}{n-1}$ for some $\frac{1}{2}<\psi\leq1$.  We return to the more general model at the end of the section.

We will consider Poisson approximation for $\xi|S_\infty|$, where $\xi$ has a Bernoulli distribution with mean $q$, independent of all else.  The event $\{\xi=0\}$ represents a catastrophe in which the entire population is infected, which happens with (small) probability $1-q$, independent of the dynamics of the epidemic model.  By the triangle inequality, we then write
\begin{multline}\label{eq:bb}
d_{TV}(\mathcal{L}(\xi|S_\infty|),\mbox{Po}(n^{1-\psi}))\leq
d_{TV}(\mathcal{L}(\xi|S_\infty|),\mathcal{L}(\xi W))
+d_{TV}(\mathcal{L}(\xi W),\mbox{Po}(\Lambda))\\
+d_{TV}(\mbox{Po}(\Lambda),\mbox{Po}(n^{1-\psi}))\,,
\end{multline}
where $W$ is the number of isolated vertices (i.e., the number of vertices that cannot be reached from any other vertex) in the random directed graph $G$ described above, and
$$
\Lambda=\sum_{i=1}^n\prod_{j\not=i}\left(1-\frac{N_j}{n-1}\right)\,,
$$
so that $\mbox{Po}(\Lambda)$ is a mixed Poisson distribution. The first term on the right-hand side of (\ref{eq:bb}) is equal to 
$$
d_{TV}(\mathcal{L}(|S_\infty|),\mathcal{L}(W))=O(n^{1-2\psi}\log n)\,,
$$ 
by the argument leading to Corollary 2.5 of \cite{bb90}.  Similarly, the final term on the right-hand side of (\ref{eq:bb}) is $O(n^{1-2\psi}\log n)$.  We use our Poisson approximation results from above to bound the middle term on the right-hand side of (\ref{eq:bb}).

Consider first the case where the $N_i$ are fixed constants.  We note that $W$ may be written as a sum of negatively related Bernoulli random variables (see Theorem 1 of \cite{bb90}).  So, taking $\phi=0$ in our Theorem \ref{thm:negdep}, we may use our Theorem \ref{thm:po}(i) (with the choices $\lambda=\mathbb{E}W=\Lambda$ and $p=q$) to get that in this case
\begin{equation}\label{eq:bb2}
d_{TV}(\mathcal{L}(\xi W),\mbox{Po}(\Lambda)) \leq (1-q)(1+\Lambda)+\frac{q}{\Lambda}\left(\Lambda-\mbox{Var}(W)\right)\,.
\end{equation}
Taking expectations on the right-hand side of (\ref{eq:bb2}), in order to derive a bound in the case where the $N_i$ are random variables, we may follow the arguments leading to Corollary 2.5 of \cite{bb90} to obtain
$$
d_{TV}(\mathcal{L}(\xi W),\mbox{Po}(\Lambda)) \leq (1-q)(1+\mathbb{E}W)+O(n^{1-2\psi}\log n)=O((1-q)n^{1-\psi}+n^{1-2\psi}\log n)\,.
$$
Hence, from (\ref{eq:bb}) we have:
\begin{proposition}
$$
d_{TV}(\mathcal{L}(\xi|S_\infty|), \mbox{Po}(n^{1-\psi}))=O((1-q)n^{1-\psi}+n^{1-2\psi}\log n)\,.
$$
\end{proposition}
In particular, if $1-q=O(n^{-\psi}\log n)$, we obtain the order $O(n^{1-2\psi}\log n)$, the same order as in Corollary 2.5 of \cite{bb90} for the usual Reed--Frost model.  That is, if the probability $1-q$ of catastrophe is small enough, it does not affect the order of the Poisson approximation bound obtained for the ultimate number of susceptible vertices remaining. 
\begin{remark}
\emph{We can also use the arguments leading to Theorem 3 of \cite{bb90} to give a bound in the more general Martin--L\"of epidemic model discussed above.  Under the assumptions of Theorem 3 of \cite{bb90}, and with $\eta_n$ defined therein, we obtain
$$
d_{TV}(\mathcal{L}(\xi|S_\infty|),\mbox{Po}(\mu^{(n)}))=O((1-q)\mu^{(n)}+\eta_n)\,,
$$
where $\mu^{(n)}=\mathbb{E}\Lambda$.}
\end{remark}

\subsection{Extremes of associated random variables}\label{sec:extremes}

We recall the following definition of association, introduced by \cite{epw67}.
\begin{definition}
The random variables $X_1,X_2,\ldots,X_n$ are associated if
\begin{equation}\label{eq:assdef}
\mathbb{E}[f(X_i,1\leq i\leq n)g(X_i,1\leq i\leq n)]\geq \mathbb{E}[f(X_i,1\leq i\leq n)]\mathbb{E}[g(X_i,1\leq i\leq n)]\,,
\end{equation}
for all increasing functions $f$ and $g$. 
\end{definition}
Association is a notion of positive dependence that we will return to in Section \ref{sec:assoc}.  In this section, our interest is in Poisson approximation for sums of associated random variables.

Suppose that $X_1,\ldots,X_n$ are (for simplicity) identically distributed, associated random variables and define
$W=\sum_{i=1}^n\II(X_i>z)$ for some $z$, so that $W$ counts the number of the $X_i$ that exceed the threshold $z$,
and write $\lambda = n\mathbb{P}(X_1>z) = \ep W$.  Poisson approximation in this situation is considered in Section 8.3 of \cite{barbour}, using the observation that $W$ is a sum of positively related Bernoulli random variables to obtain the bound
\begin{equation}\label{eq:ext1}
d_{TV}(\mathcal{L}(W),\mbox{Po}(\lambda))\leq\left(1-e^{-\lambda}\right)\left(\frac{\mbox{Var}(W)}{\lambda}-1+
\frac{2 \lambda}{n} \right)\,.
\end{equation}
We consider now the effect of some (independent) `contamination' of our sequence on this Poisson approximation result.

Suppose that $X_1,\ldots,X_n$ are identically distributed, as before.  Furthermore, let
\begin{itemize}
\item $X_1,\ldots,X_m$ be associated, for some $m\leq n$, and
\item $X_{m+1},\ldots,X_n$ be independent of $\{X_1,\ldots,X_m\}$, with arbitrary dependence among these $n-m$ random variables.
\end{itemize}
Let $W=\sum_{i=1}^m\II(X_i>z)$, $X=\sum_{i=m+1}^n\II(X_i>z)$, and $Y=W+X$.  Note that the expected number of  random variables $X_i$ exceeding the threshold $z$ is the same as before.  In light of Theorem \ref{thm:podep}, we apply Theorem \ref{thm:po}(ii) with the choices $\lambda=\mathbb{E}Y$, $p=m/n$, and $Z$ such that $\mathbb{E}Z=\mathbb{P}(X_i>z)$ (as in the analysis in \cite{barbour}) to obtain
\begin{proposition}
\begin{equation} d_{TV}(\mathcal{L}(Y),\mbox{Po}(\lambda))
\leq\left(1-e^{-\lambda}\right)\left(\frac{\mbox{Var}(Y)}{\lambda}-1+\frac{2(n-m)}{n}(\lambda+1)+\frac{2m \lambda}{n^2} \right)\,. \label{eq:ext2}
\end{equation}
\end{proposition}
We now compare the bounds (\ref{eq:ext1}) and (\ref{eq:ext2}) in a concrete example.
\begin{example}
\emph{Let $U_0,U_1,\ldots$ have independent uniform $\mbox{U}(0,1)$ distributions, and define $X_i=U_i+U_{i-1}$, $1\leq i\leq n$.  These random variables $X_i$ are associated.  In their Example 8.3.2, \cite{barbour} shows that $\lambda=n\mathbb{P}(U_i+U_{i-1}>2-\sqrt{2\lambda/n})$, if $n\geq2\lambda$. We hence follow \cite{barbour} and choose $z=2-\sqrt{2\lambda/n}$. Example 8.3.2 of \cite{barbour} also shows that
\begin{equation}\label{eq:unifeg}
\frac{\mbox{Var}(W)}{\lambda}-1<\frac{4}{3}\sqrt{\frac{2\lambda}{n}}\,,
\end{equation}
and so the bound (\ref{eq:ext1}), for the original model without contamination, becomes
\begin{equation}\label{eq:unif1}
d_{TV}(\mathcal{L}(W),\mbox{Po}(\lambda))\leq(1-e^{-\lambda})\left(\frac{4}{3}\sqrt{\frac{2\lambda}{n}}+\frac{2\lambda}{n}\right)\,.
\end{equation}
Now consider the `contaminated' model, with $X_i$ as above for $i=1,\ldots,m$.  We let $X_{m+1},\ldots,X_n$ be independent of $X_1,\ldots,X_m$, but each with the same marginal distribution as $X_1$.  Writing 
\begin{multline*}
\frac{\mbox{Var}(Y)}{\lambda}-1=\frac{\mathbb{E}W}{\lambda}\left(\frac{\mbox{Var}(W)}{\mathbb{E}W}-1\right)+\frac{\mathbb{E}X}{\lambda}\left(\frac{\mbox{Var}(X)}{\mathbb{E}X}-1\right)\\
<\frac{4}{3}\mathbb{E}W\sqrt{\frac{2}{m\lambda}}+\frac{n-m}{n}\left(\frac{\mbox{Var}(X)}{\mathbb{E}X}-1\right)\,,
\end{multline*}
where we used the inequality (\ref{eq:unifeg}), the bound (\ref{eq:ext2}) becomes
\begin{multline}\label{eq:unif2}
d_{TV}(\mathcal{L}(Y),\mbox{Po}(\lambda))\\
\leq(1-e^{-\lambda})\left(\frac{4}{3}\mathbb{E}W\sqrt{\frac{2}{m\lambda}}+\frac{n-m}{n}\left(\frac{\mbox{Var}(X)}{\mathbb{E}X}-1\right)+\frac{2(n-m)}{n}(\lambda+1)+\frac{2m\lambda}{n^2}\right)\,.
\end{multline}
Consider the case where $\lambda$ is fixed and $n\rightarrow\infty$.  In this case, the upper bound of (\ref{eq:unif1}) is of order $O(n^{-1/2})$.  If $n-m=O(\sqrt{n})$ and $(\mathbb{E}X)^{-1}\mbox{Var}(X)-1=O(n^{-1/2})$, then the upper bound (\ref{eq:unif2}) for the model with contamination is of this same order.
}
\end{example}

\section{Association and negative association}\label{sec:assoc}

We now turn our attention to the application of Theorem \ref{thm:po} to derive Poisson approximation results for sums of associated or negatively associated random variables.  We recall the definition (\ref{eq:assdef}) of association, and the following definition of negative association, introduced by \cite{jp83}.  
\begin{definition}
The random variables $X_1,\ldots,X_n$ are said to be negatively associated if
$$
\mathbb{E}[f(X_i,i\in\Gamma_1)g(X_i,i\in\Gamma_2)]\leq\mathbb{E}[f(X_i,i\in\Gamma_1)]\mathbb{E}[g(X_i,i\in\Gamma_2)]\,,
$$
for all non-decreasing functions $f$ and $g$, and all $\Gamma_1,\Gamma_2\subseteq\{1,\ldots,n\}$ such that $\Gamma_1\cap\Gamma_2=\emptyset$.  
\end{definition}
We also refer the reader to  \cite{bk00}, \cite{d13} and references therein for further discussion of the association and negative association properties, their applications, and some approximation results for sums of associated or negatively associated random variables.

We consider firstly Poisson approximation results for $Y=X_1+\cdots+X_n$, where $X_1,\ldots,X_n$ are negatively associated, non-negative integer-valued random variables.  For each $i\in\{1,\ldots,n\}$, we choose $J(i)\subseteq\{1,\ldots,n\}\setminus\{i\}$ and define $Z_i=\sum_{j\in J(i)}X_j$.  In the setting of compound Poisson approximation considered by \cite{d13}, for example, these sets $J(i)$ represent a `neighbourhood of dependence' of $X_i$, containing those indices $j$ such that $X_j$ is strongly dependent (in some sense) on $X_i$.  Here, however, we are free to make any choice of these sets $J(i)$.  In the examples we consider below, we will choose $J(i)=\emptyset$ for each $i$, for simplicity, though our Poisson approximation results apply with an arbitrary choice of these sets.

Given these sets $J(i)$ for $i=1,\ldots,n$, we define
$$
\theta_j=\frac{1}{j}\sum_{i=1}^n\mathbb{E}\left[X_i\II(X_i+Z_i=j)\right]\,,\qquad j\geq1\,,
$$
and $\theta=\sum_{j=1}^\infty\theta_j$.  Letting $\Theta$ be a random variable, independent of all else, with $\mathbb{P}(\Theta=j)=\theta_j/\theta$ for $j\geq1$, Lemma 3.1 of \cite{d13} shows that $Y^*\leq_{st}Y+\Theta^*$.  So, since there exists a coupling such that $Y^*\leq Y+\Theta^*$ almost surely, and noting that $\Theta^*\geq1$ almost surely, we may apply Theorem \ref{thm:po}(i) with the choice
$$
p=\mathbb{P}(\Theta^*=1)=\frac{\theta_1}{\sum_{j=1}^\infty j\theta_j}\,,
$$
by the definition of the size-bias distribution.  We emphasise again that this applies for any choice of the sets $J(i)$, $i=1,\ldots,n$.  For simplicity, we state explicitly in Corollary \ref{cor:negass} below the bound we obtain from Theorem \ref{thm:po}(i) in this setting with the choices $\lambda=\mu$ and $J(i)=\emptyset$ for each $i$, in which case $\theta_j=\sum_{i=1}^n\mathbb{P}(X_i=j)$ for each $j\geq1$.
\begin{corollary}\label{cor:negass}
Let $Y=X_1+\cdots+X_n$, where $X_1,\ldots,X_n$ are negatively associated, non-negative integer-valued random variables, with $\mu=\mathbb{E}Y>0$.  Let $p=\mu^{-1}\sum_{i=1}^n\mathbb{P}(X_i=1)$.  Then
$$
d_{TV}(\mathcal{L}(Y),\mbox{Po}(\mu))\leq(1-e^{-\mu})\left\{1+\mu+(1-2p)\left(\frac{\sigma^2}{\mu}+\mu\right)\right\}\,,
$$ 
where $\sigma^2=\mbox{Var}(Y)$.
\end{corollary}
\begin{remark}
\emph{In the setting of Corollary \ref{cor:negass}, if the $X_i$ are Bernoulli random variables then we have $p=1$.  Since negatively associated Bernoulli random variables are known to be negatively related (see page 78 of \cite{e05}, for example), we know that $Y^*\leq_{st}Y+1$ in this setting, and so the results of \cite{dlu12} may be applied.  In this case, our Corollary \ref{cor:negass} gives the same bound as \cite{dlu12}.}
\end{remark}
Given the above remark, we may think of $p$ as measuring (in a certain sense) how close the $X_i$ are to having Bernoulli distributions, with an increasing $p$ resulting in a smaller upper bound in Poisson approximation.  We illustrate this with a simple example.
\begin{example}
\emph{In the setting of Corollary \ref{cor:negass}, if $X_1,\ldots,X_n$ are identically distributed with $\mathbb{P}(X_1=0)=1-a$, $\mathbb{P}(X_1=1)=a-\epsilon$ and $\mathbb{P}(X_1=2)=\epsilon$ for some $a\geq\epsilon\geq0$ with $a+\epsilon<1$, then we have $p=\frac{a-\epsilon}{a+\epsilon}$, and obtain a good Poisson approximation bound when $\epsilon$ is small.}
\end{example}

We have so far discussed only the approximation of sums of negatively associated random variables.  We now turn our attention to sums of associated random variables, where we use similar techniques to the above.  We let $X_1,\ldots,X_n$ be associated, non-negative integer-valued random variables, and define $\theta_j$ (for $j\geq1$) as above, again with any choice of the sets $J(i)\subseteq\{1,\ldots,n\}\setminus\{i\}$ allowed for each $i$.  We also let $\Theta$ be as above, again independent of all else, and write $Y=X_1+\cdots+X_n$.

We further define the random variable $V$, independent of all else, with $\mathbb{P}(V=i)=\mathbb{E}X_i/\mathbb{E}Y$ for $i=1,\ldots,n$.  Write $Z=X_V+Z_V$.  By Lemma 3.2 of \cite{d13}, we have that $Y^*\geq_{st}Y+\Theta^*-Z$, and so, analogously to the case of negative association, we may apply Theorem \ref{thm:po}(ii) with this choice of $Z$ and with $p=\theta_1\left(\sum_{j=1}^\infty j\theta_j\right)^{-1}$.  An analogue of Corollary \ref{cor:negass} thus also applies in this setting.

As before, if the $X_i$ are Bernoulli random variables then we have $p=1$.  Since associated Bernoulli random variables are positively related (see, for example, page 77 of \cite{e05}), the results of \cite{dlu12} may also be applied here, and we obtain the same bound as \cite{dlu12} in this special case.  Again, we are not limited to considering Bernoulli random variables, and may use the value of $p$ to measure how close the $X_i$ are to being Bernoulli. 

\subsection{Application to simple random sampling}

Let $c_1,\ldots,c_n$ be $n$ (not necessarily distinct) non-negative integers.  Suppose we take a random sample of size $m<n$ without replacement from this collection of numbers, and let $X_1,\ldots,X_m$ denote this sample.  The random variables $X_1,\ldots,X_m$ are negatively related (see Section 3.2 of \cite{jp83}), and we will consider Poisson approximation of $Y=X_1+\cdots+X_m$ using Corollary \ref{cor:negass}.  Straightforward calculations give $\mu=(m/n)\sum_{i=1}^nc_i$ and
$$
\sigma^2=\frac{m}{n}\sum_{i=1}^nc_i^2+\frac{m(m-1)}{n(n-1)}\sum_{i=1}^n\sum_{j\not=i}c_ic_j-\frac{m^2}{n^2}\left(\sum_{i=1}^nc_i\right)^2\,.
$$
Noting that, by exchangeability, $\mathbb{P}(X_i=1)=n^{-1}|\{i:c_i=1\}|$ for each $i$, we may take
$$
p=\frac{|\{i:c_i=1\}|}{\sum_{i=1}^nc_i}\,,
$$
and apply Corollary \ref{cor:negass} with these choices.

In the case where each $c_i$ is either 0 or 1, $Y$ has a hypergeometric distribution, for which good Poisson approximation bound are well-known; see, for example, Theorem 6.A of \cite{barbour} for an upper bound obtained using the Stein--Chen method.  This bound is obtained by writing $Y$ as a sum of negatively related Bernoulli random variables.  Note that our result generalises this bound: in the case where each $c_i$ is either 0 or 1, we may take $p=1$, and we recover the upper bound given by \cite{barbour}. 

\section{A Poincar\'e inequality}\label{sec:poin}

Next, we show how the  assumptions of Section \ref{sec:pois} may be employed to prove a Poincar\'e inequality, relaxing strict monotonicity assumptions.

\begin{definition}
Define the discrete Poincar\'e constant $R_W$ for a (non-negative, integer-valued) random variable $W$ by
$$
R_W=\sup_{g\in\mathcal{G}(W)}\left\{\frac{\mathbb{E}[g(W)^2]}{\mathbb{E}[\Delta g(W)^2]}\right\}\,,
$$
where the supremum is taken over the set
$$
\mathcal{G}(W)=\{g:\mathbb{Z}^+\mapsto\mathbb{R}\mbox{ with }\mathbb{E}[g(W)^2]<\infty\mbox{ and }\mathbb{E}[g(W)]=0\}\,.
$$
\end{definition}
We note the well-known lower bound
\begin{equation}\label{eq:plb}
R_W\geq\mbox{Var}(W)\,,
\end{equation}
obtained by choosing $g(x)=x-\mathbb{E}W$.

Theorem 1.1 of \cite{dj13} proves that if $W$ satisfies (\ref{eq:ord}), then $R_W\leq\mathbb{E}W$.  In Theorem \ref{thm:poincare} below, we weaken this condition, assuming only the conditions of Theorem \ref{thm:po}(i), to prove an analogous result.  The bound can be expressed in terms of the failure rate of $Y$, defined below.

\begin{definition} For a discrete random variable $Y$, define the failure (or hazard) rate
$$
h_Y(j)=\frac{\mathbb{P}(Y=j)}{\mathbb{P}(Y\geq j)}\,,
$$
and write $h^*_Y = \inf_{j} h_Y(j)$, where the infimum is taken over the support of $Y$. 
\end{definition}
Throughout this section, we let $I_p\sim\mbox{Be}(p)$ have a Bernoulli distribution with mean $p$, independent of all else.
\begin{theorem}\label{thm:poincare}
Let $Y$ be a non-negative, integer-valued random variable with $\mathbb{E}Y=\mu>0$.  Let $p\in(0,1]$ be such that $I_pY^*\leq_{st}Y+1$, then
$$
R_Y \leq \mu \left( 1 + \frac{1-p}{p h^*_Y} \right)\,.
$$
\end{theorem}
\begin{proof}
Our argument is based on that of \cite{dj13}.  As there, we also employ the kernel function defined by Klaassen \cite{k85}:
$$
\chi(i,j)=I(\lfloor\mu\rfloor\leq j<i)-I(i\leq j<\lfloor\mu\rfloor)-(\mu-\lfloor\mu\rfloor)I(j=\lfloor\mu\rfloor)\,.
$$
Then Lemma 5.2 of \cite{dj13} gives, for $g\in\mathcal{G}(Y)$,
\begin{multline}\label{eq:poincare}
\mathbb{E}[g(Y)^2]\leq\mu\sum_{j=0}^\infty\Delta g(j)^2\left[\mathbb{E}\chi(Y^*,j)-\mathbb{E}\chi(Y,j)\right]\\
=\mu\sum_{j=0}^\infty\Delta g(j)^2\left[\mathbb{E}\chi(Y^*,j)-\mathbb{E}\chi(I_pY^*,j)\right]
+\mu\sum_{j=0}^\infty\Delta g(j)^2\left[\mathbb{E}\chi(I_pY^*,j)-\mathbb{E}\chi(Y,j)\right]\,.
\end{multline}

Following the argument on page 517 of \cite{dj13}, using the stochastic ordering assumption we make here, the second term on the right-hand side of (\ref{eq:poincare}) may be bounded by $\mu\mathbb{E}[\Delta g(Y)^2]$. 

For the first term on the right-hand side of (\ref{eq:poincare}), we have that
$$
\mathbb{E}\chi(Y^*,j)-\mathbb{E}\chi(I_pY^*,j)=(1-p)[\mathbb{E}\chi(Y^*,j)-\mathbb{E}\chi(0,j)]=(1-p)\mathbb{P}(Y^*>j)\,,
$$
where this final equality follows from Eq$.$ (15) of \cite{dj13}.  Again employing our stochastic ordering assumption, we obtain
$$
\mathbb{E}\chi(Y^*,j)-\mathbb{E}\chi(I_pY^*,j)\leq\frac{1-p}{p}\mathbb{P}(Y\geq j)\,.
$$

We therefore have the bound
$$
\frac{\mathbb{E}[g(Y)^2]}{\mathbb{E}[\Delta g(Y)^2]} \leq \mu\left(1+\frac{1-p}{p}\left[\frac{\sum_{j=0}^\infty\Delta g(j)^2\mathbb{P}(Y\geq j)}{\sum_{j=0}^\infty\Delta g(j)^2\mathbb{P}(Y=j)}\right]\right)\,.
$$
The theorem then follows.
\end{proof}
\begin{remark}
\emph{In the case $p=1$ we obtain the bound of Theorem 1.1 of \cite{dj13}, as we would expect.
If there is a coupling of $(Y,Y^*)$ and $p\in(0,1]$ such that (\ref{eq:ndcond}) holds for all $x$, then the proof of Theorem \ref{thm:po}(i) shows that the assumptions of Theorem \ref{thm:poincare} above hold, and we have our upper bound on $R_Y$.}
\end{remark}
\begin{example}\label{eg:poincare}
\emph{We return to the setting of Example \ref{eg:po}, and let $Y=I_pZ$.  Then $\mu=p\lambda$ and 
$$
h_Y(0) =
\frac{\mathbb{P}(Y=0)}{\mathbb{P}(Y\geq0)}=\frac{(1-p)e^\lambda+p}{e^\lambda}\,.
$$  
Since $h_Y(j)=h_Z(j)$ for all $j\geq1$, we may use the increasing failure rate (IFR) property of the Poisson distribution to obtain the following bound from Theorem \ref{thm:poincare}:
\begin{equation} \label{eq:poincare2}
R_Y \leq \lambda\left(p+(1-p)\max\left\{\frac{e^\lambda-1}{\lambda},\frac{e^\lambda}{(1-p)e^\lambda+p}\right\}\right)\,.
\end{equation}
Expanding \eqref{eq:poincare2} in $\lambda$, we obtain
$ R_Y \leq \max \left( \lambda + \frac{(1-p)}{2} \lambda^2 + O(\lambda^3),  \lambda + p(1-p) \lambda^2 + O(\lambda^3) \right).$
In this case, the lower bound (\ref{eq:plb}) becomes $R_Y\geq p\lambda+p(1-p)\lambda^2$, showing that 
\eqref{eq:poincare2} is close to sharp for $p$ close to 1 and $\lambda$ small.
}
\end{example}

Applying Theorem \ref{thm:poincare} relies on both controlling the failure rate of $Y$ and finding a suitable $p$.  We conclude this section by showing that we can bound $p$, and hence $R_Y$, under the standard $c$-log-concavity condition (see \cite{cap09}).

\begin{corollary}\label{cor:poincare}
Let $Y$ be a non-negative, integer-valued random variable with $\mathbb{E}Y=\mu>0$.  Assume that there exists  $c >0$ such that
\begin{equation}\label{eq:clogconc}
\frac{ \pr(Y= k)^2 - \pr(Y= k+1) \pr(Y=k-1)}{ \pr(Y= k) \pr(Y=k+1)} = \frac{ \pr(Y=k)}{\pr(Y=k+1)} -  \frac{ \pr(Y=k-1)}{\pr(Y=k)}
\geq c
\end{equation}
for all $k\geq0$. Then
$$
R_Y \leq \frac{1}{c} \left( 1 + (1 - c \mu) \frac{ \pr(Y \geq 1)}{\pr(Y=0)} \right)\,.
$$
\end{corollary}
\begin{proof}
We first show that, if $p=\mu c$, $I_pY^*\leq_{st}Y+1$.  With this, we will then see that the bound follows from Theorem \ref{thm:poincare}.  That is, we begin by showing that, as in \eqref{eq:tails}, $p\mathbb{P}(Y^*\geq j)\leq\mathbb{P}(Y+1\geq j)$
for all $j\geq1$, i.e.,
\begin{equation} \label{eq:needed}
p\leq\inf_{k\geq0}\left\{\frac{\mathbb{P}(Y\geq k)}{\mathbb{P}(Y^*\geq k+1)}\right\}\,.
\end{equation}
We observe that, summing the collapsing sum in \eqref{eq:clogconc} from $k = 0$ to $\ell-1$, we obtain:
\begin{equation}\label{eq:ratio}
\frac{\pr(Y= \ell-1)}{\pr(Y=\ell)}\geq c \ell \,,
\end{equation}
for all $\ell\geq0$. Fixing some $k\geq0$, we have by \eqref{eq:ratio} that
\begin{eqnarray*}
\mathbb{P}(Y^*\geq k+1)&=& \sum_{\ell =k+1}^\infty \pr(Y^* = \ell) \\
& = & \frac{1}{\mu} \sum_{\ell = k+1}^\infty \ell \pr(Y = \ell) \\
& \leq & \frac{1}{\mu c} \sum_{\ell =k+1}^\infty \pr(Y = \ell-1) =  \frac{1}{\mu c} \pr(Y \geq k),
\end{eqnarray*}
and the result \eqref{eq:needed} follows with $p = \mu c$.

Now, \eqref{eq:clogconc} implies that $Y$ is log-concave, which in turn implies the IFR property, so that $h_Y^* = 
h_Y(0) = \pr(Y= 0)$. Substituting this into Theorem \ref{thm:poincare} we obtain our upper bound on $R_Y$. 

\end{proof}

Notice that if $Y$ is Poisson with mean $\mu$, we can take $c = 1/\mu$ in \eqref{eq:clogconc}, to recover the
fact that $p = 1$ and deduce the standard Poisson Poincar\'{e} inequality $R_Y \leq \mu$.
Note also that in \cite{joh17}, the (stronger) fact that $R_Y \leq 1/c$ is proved under the $c$-log-concavity condition \eqref{eq:clogconc}, but this requires use of the discrete Bakry-\'{E}mery theory of \cite{cap09}.

\section{Normal approximation}\label{sec:norm}

Finally, we show how our assumptions  will carry over to 
normal approximation, again based on ideas used in Stein's method.  Several different coupling constructions are used in Stein's method for normal approximation; see \cite{cgs11} for an introduction to these, and to the area more generally.  In line with our work in Section \ref{sec:pois}, we will consider normal approximation using the size-biased coupling; again we will consider a
non-negative random variable $W$ and write $W^*$ for its  size-biased version. 
 In this setting, a natural assumption under which normal approximation results for $W$ have been established is boundedness of $W^*$; see Theorems 5.6 and 5.7 of \cite{cgs11}, for example.  

We  prove an analogous normal approximation theorem which allows the random variable $W$ to be contaminated with some independent noise, as we did in the Poisson approximation setting in Section \ref{sec:noise}.  
Our bound will be stated in terms of the Kolmogorov distance, defined by
$$
d_K(\mathcal{L}(Y),\mathcal{L}(Z))=\sup_{z\in\mathbb{R}}|\mathbb{P}(Y\leq z)-\mathbb{P}(Z\leq z)|\,,
$$ 
and for a non-negative random variable $W$ we will write
$$
D_W=\mathbb{E}\left|\mathbb{E}\left[1-\frac{\mathbb{E}W}{\mbox{Var}(W)}(W^*-W)\bigg|W\right]\right|\,.
$$

Before stating the main theorem of this section, we note that for any non-negative random variable $W$, $W\leq_{st}W^*$, so that it is always possible to couple $W$ and $W^*$ such that $W\leq W^*$ almost surely.  We also note that the existence of a size-biased coupling such that $W\leq W^*\leq W+c$ is equivalent to the assertion that $W^*\leq_{st}W+c$; see Section 7 of \cite{ab15}.
\begin{theorem}\label{thm:normal}
Let $W$ and $X$ be independent, non-negative random variables, with $0\leq X\leq a$ a$.$s$.$ for some $a>0$.  Assume also that 
\begin{itemize}
\item $W$ and $W^*$ are coupled such that $W\leq W^*\leq W+c$ for some $c>0$, and
\item $X$ and $X^*$ are coupled such that $X\leq X^*$.
\end{itemize}
Let $Y=W+X$, $\mu=\mathbb{E}Y$ and $\sigma^2=\mbox{Var}(Y)$.  Define $p=\mathbb{E}W/\mu$ and
$$
\widetilde{Y}=\frac{Y-\mu}{\sigma}\,.
$$ 
Then
$$
d_K(\mathcal{L}(\widetilde{Y}),\mbox{N}(0,1))\leq \frac{\mbox{Var}(W)}{\sigma^2}D_W+\frac{\mbox{Var}(X)}{\sigma^2}D_X+0.82\frac{c^2\mu}{\sigma^3}+\frac{c}{\sigma}+\frac{\mu}{\sigma^2}a(1-p)\,.
$$
\end{theorem}
\begin{proof}
Abusing notation, let $\widetilde{Y}^*=\frac{Y^*-\mu}{\sigma}$ and, for fixed $z\in\mathbb{R}$, let $f:\mathbb{R}\mapsto\mathbb{R}$ be the solution to the Stein equation $f^\prime(w)-wf(w)=\II(w\leq z)-\Phi(z)$, where $\Phi$ is the distribution function of the standard normal distribution.  Note the standard bound $|f^\prime(w)|\leq1$ (see Lemma 2.3 of \cite{cgs11}, for example).

Now, following the proof of Theorem 5.7 of \cite{cgs11}, we write
\begin{align}
\mathbb{P}(\widetilde{Y}\leq z)-\Phi(z) & = \mathbb{E}\left[f^\prime(\widetilde{Y})-\widetilde{Y}f(\widetilde{Y})\right] \notag \\
 & = \mathbb{E}\left[f^\prime(\widetilde{Y})\left(1-\frac{\mu}{\sigma}(\widetilde{Y}^*-\widetilde{Y})\right)-\frac{\mu}{\sigma}\int_0^{\widetilde{Y}^*-\widetilde{Y}}\left(f^\prime(\widetilde{Y}+t)-f^\prime(\widetilde{Y})\right)\,dt\right]\,.
\label{eq:n1}
\end{align}
The absolute value of the final term on the right-hand side of (\ref{eq:n1}) is bounded, as in Theorem 5.7 of \cite{cgs11}, by
$$
0.82\frac{c^2\mu}{\sigma^3}+\frac{c}{\sigma}+\frac{\mu}{\sigma^2}\mathbb{E}\left[(Y^*-Y)\II(Y^*-Y>c)\right]\,.
$$
Note that we can write
$Y^*=I_p(W^*+X)+(1-I_p)(W+X^*)$, where $I_p$ is a Bernoulli variable with mean $p$ independent of all else; see Corollary 2.1 of \cite{cgs11}. Conditioning on $I_p$, we then have 
$$
\mathbb{E}\left[(Y^*-Y)\II(Y^*-Y>c)\right]\leq a(1-p)\,,
$$
by the assumptions of our theorem.

We use our representation of $Y^*$, and the independence of $W$ and $X$, to write the first term on the right-hand side of (\ref{eq:n1}) as
\begin{multline*}
\frac{\mbox{Var}(W)}{\sigma^2}\mathbb{E}\left[f^\prime(\widetilde{Y})\left(1-\frac{\mu}{\mbox{Var}(W)}I_p(W^*-W)\right)\right]\\
+\frac{\mbox{Var}(X)}{\sigma^2}\mathbb{E}\left[f^\prime(\widetilde{Y})\left(1-\frac{\mu}{\mbox{Var}(X)}(1-I_p)(X^*-X)\right)\right]\,.
\end{multline*}
Conditioning on $W$ and $X$, and applying the bound $|f^\prime(w)|\leq1$, this may be bounded by $\sigma^{-2}\left(\mbox{Var}(W)D_W+\mbox{Var}(X)D_X\right)$.
\end{proof}

\subsection{Application: the lightbulb process}

Consider the following model, motivated by a pharmaceutical study of dermal patches designed to activate particular receptors, though often phrased in terms of lightbulbs being switched on and off;  see \cite{gz11} and references therein.  We begin with $n$ lightbulbs, all switched off.  At time $r$ (for $r=1,\ldots,n$), exactly $r$ of the $n$ lightbulbs are chosen, uniformly at random, and their state switched.  One random variable of interest is $W$, the number of lightbulbs switched on after time $n$.

Goldstein and Zhang \cite[Theorem 1.1]{gz11}  prove a bound for normal approximation of $W$.  Combining their bound with our Theorem \ref{thm:normal}, we see the effect of `contaminating' $W$ by $X$, which (for simplicity) we define to have a $\mbox{Bin}(k,\alpha)$ distribution, independent of $W$.  Other types of contamination can also be investigated in this framework.

Also for simplicity, we will restrict attention to the case where $n$ is even.  In this case we have $\mathbb{E}W=n/2$, and we let $\tau^2$ denote the variance of $W$, which is equal to $(n/4)(1+O(e^{-n}))$; see \cite{gz11} for further details, and for a discussion of the differences between the cases where $n$ is even and $n$ is odd.  Goldstein and Zhang \cite{gz11} also provide a coupling such that $W\leq W^*\leq W+2$ a$.$s$.$, and show that 
$$
D_W\leq\frac{n}{2\tau^2}\left(\frac{1}{2\sqrt{n}}+\frac{1}{2n}+\frac{1}{3}e^{-n/2}\right)\,.
$$   
Straightforward calculations (in the spirit of Section 5.3 of \cite{cgs11}) show that $D_X\leq\sqrt{\frac{\alpha}{(1-\alpha)k}}$.  Hence, letting $Y=W+X$, the bound of Theorem \ref{thm:normal} becomes
\begin{proposition}
\begin{multline*}
d_K(\mathcal{L}(\widetilde{Y}),\mbox{N}(0,1))\leq \frac{1}{\sigma^2}\left\{\frac{n}{2}\left(\frac{1}{2\sqrt{n}}+\frac{1}{2n}+\frac{1}{3}e^{-n/2}\right)+\alpha\sqrt{\alpha(1-\alpha)k}\right\}\\
+\frac{1.64n}{\sigma^3}+\frac{2}{\sigma}+\frac{\alpha k^2n}{\sigma^2(n+2\alpha k)}\,,
\end{multline*}
where $\sigma^2=\tau^2+\alpha(1-\alpha)k$.
\end{proposition}
So, for example, if $\alpha=O(1)$ and $k=O(n^{-1/4})$ this bound is of the same order, $O(n^{-1/2})$, as the bound on $d_K(\mathcal{L}(\widetilde{W}),\mbox{N}(0,1))$ given by Theorem 1.1 of \cite{gz11}, where $\widetilde{W}=\frac{W-n/2}{\tau}$.

\appendix

\section{Proof of Lemma \ref{lem:key}}\label{sec:lem}

Let $p_Y(j)=\mathbb{P}(Y=j)$.  Substituting the Stein--Chen equation \eqref{eq:stein}, we have
\begin{eqnarray}
\pr( Y \in A) - \Pi_\lambda(A) & = & \sum_{j=0}^\infty p_Y(j) \left( \II( j \in A) - \Pi_{\lambda}(A) \right) \nonumber \\
& = & \sum_{j=0}^\infty p_Y(j) \left( 
\lambda g_A(j+1) - j g_A(j)  \right) \nonumber \\
& = & \lambda \sum_{j=0}^\infty p_Y(j)   g_A(j+1) - \mu \sum_{j=0}^\infty  p_{Y^*}(j) g_A(j)   \label{eq:todo} \\
& = & \sum_{k=0}^\infty \left( \lambda \ol{P}_Y(k)    - \mu  \ol{P}_{Y^*}(k+1)  \right) \Delta g_A(k)\,,   \nonumber
\end{eqnarray}
where we deal with the first term in (\ref{eq:todo}) since for any function $h$ with $h(0) = 0$, summation by parts gives
$$ \sum_{j=0}^\infty p_Y(j) h(j+1) =  \sum_{j=0}^\infty p_Y(j) \left( \sum_{k=0}^j \Delta h(k) \right) = \sum_{k=0}^\infty \Delta h(k) \sum_{j=k}^\infty p_Y(j)
= \sum_{k=0}^\infty  \Delta h(k) \ol{P}_{Y}(k).$$
The result follows taking $h = g_A$, since we know $g_A(0) = 0$ by assumption.

\end{document}